\documentclass[]{interact}

\usepackage{epstopdf}
\usepackage[caption=false]{subfig}

\usepackage[numbers,sort&compress]{natbib}
\usepackage{color}
\bibpunct[, ]{[}{]}{,}{n}{,}{,}
\makeatletter
\def\NAT@def@citea{\def\@citea{\NAT@separator}}
\makeatother

\theoremstyle{plain}
\newtheorem{theorem}{Theorem}[section]
\newtheorem{lemma}[theorem]{Lemma}
\newtheorem{corollary}[theorem]{Corollary}
\newtheorem{proposition}[theorem]{Proposition}

\theoremstyle{definition}
\newtheorem{definition}[theorem]{Definition}
\newtheorem{example}[theorem]{Example}

\theoremstyle{remark}
\newtheorem{remark}{Remark}

\usepackage{graphicx}

\DeclareMathOperator{\eco}{e-conv}					

\DeclareMathOperator{\conv}{conv}
\DeclareMathOperator{\cone}{cone}

\DeclareMathOperator{\cl}{cl}

\DeclareMathOperator{\epco}{\rm{e}^{\prime} -conv} 	
\DeclareMathOperator{\ep}{\text{e}^{\prime}} 	
\DeclareMathOperator{\dom}{dom}
\DeclareMathOperator{\epi}{epi}

\def\supp{\mathop{\rm sup}}

\newcommand{\en}{\rightarrow}											
\newcommand{\R}{\mathbb{R}}												
\newcommand{\Ramp}{\overline{\R}}					

\newcommand{\xb}{\overline{x}}										

\newcommand{\ci}{\left\langle}										
\newcommand{\cd}{\right\rangle}										

\begin{document}

\articletype{ARTICLE TEMPLATE}

\title{On subdifferentials via a generalized conjugation scheme: an application to DC problems  and optimality conditions}

\author{
\name{M.~D. Fajardo\textsuperscript{a}\thanks{CONTACT M.~D. Fajardo. Email: md.fajardo@ua.es} and J. Vidal \textsuperscript{b}\thanks{J. Vidal. Email: j.vidal@uah.es}}
\affil{\textsuperscript{a}University of Alicante, Spain; \textsuperscript{b}Department of Physics and Mathematics, University of Alcal\'a de Henares, Madrid, Spain}
}

\maketitle

\begin{abstract}
This paper studies properties of a subdifferential defined using a generalized conjugation scheme. We relate this subdifferential together with the domain of an appropriate conjugate function and the $\varepsilon$-directional derivative. In addition, we also present necessary conditions for $\varepsilon$-optimality and global optimality in optimization problems involving the difference of two convex functions. These conditions will be written via this generalized notion of subdifferential studied in the first sections of the paper.
\end{abstract}

\begin{amscode}
	52A20, 26B25, 90C26, 90C46 
\end{amscode}


\section{Introduction}
\label{sec:1}

Among the huge variety of optimization problems that can be found in real life, those whose objective function is expressed as the difference of two convex functions have attained a lot of attention since decades in the optimization community. These problems are called DC problems where DC means difference of convex functions. For an in-depth introduction as well as some applications of DC programming, we recommend the reader \cite{HT1999,CGR2018,AT2005,TA1997} and the references therein.

Given a general DC problem, there exist different approaches to study conditions for optimality. Just to mention a few, we start with the renewed paper \cite{CLA2021}, where the authors deal with optimality conditions which are necessary and sufficient  for DC semi-infinite programming using subdifferentials. In \cite{D2020} new global conditions for non-smooth DC optimization problems via affine support sets are developed, \cite{JG1996} works with DC problems under convex inequality constraints and \cite{ET2006} explores DC programming in reflexive Banach spaces. The works of \cite{DMD2009} and \cite{FZ2014} develop conditions in terms of epigraphs with infinite constraints, while \cite{JLi2011} focuses on polynomial constraints and \cite{SFu2014,Z2013} on DC programs involving composite functions and canonical DC problems, respectively. We also mention \cite{FML1998} where a group of global optimality conditions is compared using a generalized conjugation theory as framework. 

Concerning optimality conditions for global maximum of a general function in Euclidean spaces, we mention \cite{HU2001} and \cite{DHL1998}. While the former paper introduces a useful characterization of global optimality using the subdifferential and the normal cone, the latter goes beyond that. More precisely, it compares the characterization coming from \cite{HU2001} with some other necessary and sufficient conditions named Strekalovski, Singer-Toland and Canonical-dc-programming passing through their necessary assumptions to hold and their relationships.

In the current paper we will not be motivated directly by any of these conditions, but by another from \cite{HU1988} written in terms of just subdifferentials: if $f$ and $g$ are proper convex and lower semicontinuous functions, $\xb$ is a global minimizer of $f=g-h$ if and only if 
\begin{equation*}
	\partial_\varepsilon h(\xb)\subseteq \partial_\varepsilon g(\xb),~~~\text{ for all }\varepsilon\geq 0.
\end{equation*}
This subdifferential notion is linked to Fenchel conjugation scheme, in fact it is called Fenchel subdifferential in \cite{Z2002}. An evidence of this connection is the following equivalence between the subdifferential of a convex and lower semicontinuous function and the subdifferential of its Fenchel conjugate 
\begin{equation}\label{eq:FEN}
x^* \in \partial_\varepsilon f( x) \text{ if and only if } x \in \partial_\varepsilon f^* (x^*), ~~~\text{ for all } \varepsilon \geq 0.
\end{equation} 

Equivalence (\ref{eq:FEN}) holds thanks to Fenchel-Moreau theorem, which establishes the equality between a proper convex lower semicontinuous function and its Fenchel biconjugate.
Nevertheless, Fenchel conjugation scheme is not suitable (in the sense that Fenchel-Moreau theorem does not hold with them)  for a class of functions which generalizes the class of convex and lower semicontinuous functions, namely the  \emph{evenly convex functions} (see \cite{RVP2011}). These functions have epigraphs which are \emph{evenly convex sets}, i.e.,  intersections of arbitrary families (possibly empty) of open half-spaces. This kind of sets was initially defined by Fenchel \cite{F1952} in the Euclidean space, trying to extend the polarity theory to nonclosed convex sets. Later, they were applied in linear inequality systems (see \cite{GJR2003} and \cite{GR2006}), because evenly convex sets are the solution sets of linear systems containing strict inequalities. In addition, \cite{KMZ2007} contains basic properties of evenly convex sets expressed in terms of their sections and projections.

In \cite{MLVP2011} it is provided a conjugation scheme for extended real functions, called \emph{$c$-conjugation}, and a subdifferential notion associated with it, which would allow to obtain a counterpart of (\ref{eq:FEN}) for proper evenly convex functions; see \cite{FGV2021}. Another interesting application of $c$-conjugation developed in the last years has been the building of different dual problems for a primal convex one, in which strong duality property is related to the  even convexity of the functions in the primal problem. 
A very recent monograph presents the state of the art in Even Convexity and Optimization; see \cite{FGRVP2020}.

Our commitment in this paper is to develop basic, on one hand, and interesting, on the other hand, properties of the subdifferential concept associated with $c$-conjugation, and to obtain optimality conditions for DC problems via this operator.

Concerning the organization,  Section \ref{sec:2} summarizes the necessary results throughout the paper. In Section \ref{sec:3} we introduce the formal definition of the subdifferential of interest in this paper and its main properties showing its connection with a generalized notion of conjugate function. Section \ref{sec:4} is devoted to the analysis of deeper properties of this subdifferential. In particular, we will present its relationship with the domain of the conjugate function and the subdifferential of its $\varepsilon$-directional derivative. Section \ref{sec:5} develops necessary  optimality conditions for DC problems when the involved functions are proper and evenly convex. Finally, Section \ref{sec:6} summarizes the most important achievements of this paper as well as points out related open problems for future research.

\section{Preliminaries}
\label{sec:2}

We denote by $X$ a nontrivial separated locally convex space, lcs in short, equipped with the $\sigma(X,X^*)$ topology induced by $X^*$. Here, $X^*$ represents the continuous dual space of $X$ endowed with the $\sigma(X^*,X)$ topology. Given a continuous linear functional $x^*\in X^*$, $\ci x,x^*\cd$ represents its value at $x\in X$. If $D\subseteq X$, $\conv D$ and $\cl D$ stand for its convex hull and closure, respectively.
 
As we said in the previous section, Fenchel \cite{F1952} defined an evenly convex set (e-convex set, in brief) as an intersection of an arbitrary family (possibly empty) of open half-spaces. The following equivalent definition provides a very useful tool in order to identify e-convex sets.

\begin{definition}[{\cite[Def. 1]{DML2002}}] \label{defeconv}
A set $C\subseteq X$ is \emph{e-convex} if for every point $x_0\notin C$, there exists $x^*\in X^*$ such that $\ci x-x_0,x^*\cd<0$, for all $x\in C$. 
\end{definition}

Given $C\subseteq X$, we represent the smallest e-convex set in $X$ containing $C$, i.e., its \emph{e-convex hull}, by $\eco C$. If $C\subseteq X$ is convex, the inclusions $C\subseteq \eco C \subseteq \cl C$ are fulfilled. It is worthwhile adding that due to the fact that the class of e-convex sets is closed under arbitrary intersections, this operator is well defined.
By hypothesis $X$ is a separated lcs, so $X^*\neq \{0\}$. Furthermore, Hahn-Banach theorem implies not only that $X$ is e-convex, but also the fact that every closed or open convex set is e-convex, too.

If $f:X\en \Ramp$, the set $\dom f=\{ x\in X$ : $ f(x)<+\infty \}$ represents its \emph{effective domain} while $\epi f = \left\{ \left( x,r\right) \in X\times \R \, : \, f(x)\leq r\right\}$ stands for its \emph{epigraph}. The function $f$ is \emph{proper} if $f(x)>-\infty$ for all $x\in X$ and $\dom f\neq \emptyset$.
With $\cl f$ we mean the \emph{lower semicontinuous hull} of $f$, i.e., the function verifying the equality $\epi (\cl f) = \cl(\epi f)$.
We say that $f$ is \emph{lower semicontinuous}, or lsc, if $f(x)=\cl f(x)$ for all $x\in X$, and \textit{e-convex} if $\epi f$ is e-convex in the product space $X\times \R$.
It is immediate to observe that the class of lsc convex functions is contained in the class of e-convex functions. Nevertheless, this inclusion fails to be an equality between sets.
\begin{example}[{\cite[Ex. 2.1]{FV2017}}]
Let $f \colon \R \to \Ramp$ be the function
\begin{equation*} 
	f(x) =\left\{
	\begin{array}{ll}
		x,&  \text{ if } x >0,\\[0.2cm]
		+\infty,&  \text{ otherwise.}%
	\end{array}
	\right.
\end{equation*}
It is straightforward to see that
$$\epi f=\{(x,\alpha)\in \R^2:x>0,\alpha \geq x\}$$
is not closed but convex. However, it is e-convex, since for any point $(0,\alpha)$ with nonnegative $\alpha$, the hyperplane 
$H=\left \{ (x,y) \in \R^2:x=0 \right \} $ passes through the point with empty intersection with $\epi f$; recall Definition \ref{defeconv}.
\end{example}

We continue defining the \emph{e-convex hull} of a given function $f:X\rightarrow \overline{\R}$ as the largest e-convex minorant of $f$ and we denote it by $\eco f$.
Using the generalized convex conjugation theory presented in Moreau \cite{Mor1970}, a conjugation scheme appropriated for e-convex functions is given in \cite{MLVP2011}.
Let $W:=X^{\ast }\times X^{\ast }\times \R$ and \emph{the coupling functions }$c:X\times W\rightarrow \overline{\R}$ and $c^{\prime }:W\times X\rightarrow \overline{\R}$ defined by
\begin{equation} \label{equation: def c and c prime}
c(x,(x^{\ast },u^{\ast },\alpha ))=c^{\prime }\left( (x^{\ast },u^{\ast},\alpha ),x\right) :=\left\{
\begin{array}{ll}
\left\langle x,x^{\ast }\right\rangle & \text{if }\left\langle x,u^{\ast}\right\rangle <\alpha, \\
+\infty & \text{otherwise.}
\end{array}
\right.
\end{equation}
If $f:X\rightarrow \overline{\R}$ and $g:W\rightarrow \overline{\R}$ are two given functions, the \emph{$c$-conjugate} of $f$, $f^{c}:W\rightarrow \overline{\R}$, and the \emph{c$^{\prime }$-conjugate} of $g$, $g^{c^{\prime}}:X\rightarrow \overline{\R}$, are as follows
\begin{align*}
f^{c}(x^{\ast },u^{\ast },\alpha) &:= \supp_{x\in X}\left\{ c(x,(x^{\ast},u^{\ast },\alpha ))-f(x)\right\},\\
 g^{c^{\prime }}(x) &:= \supp_{(x^{\ast },u^{\ast },\alpha )\in W}\left\{c^{\prime }\left( (x^{\ast },u^{\ast },\alpha ),x\right) -g(x^{\ast},u^{\ast },\alpha )\right\},
\end{align*}
 with the sign conventions $\left( +\infty \right) +\left( -\infty \right)=\left( -\infty \right) +\left( +\infty \right) =\left( +\infty \right)-\left( +\infty \right)$ $=\left( -\infty \right) -\left( -\infty \right)=-\infty.$

In \cite{MLVP2011} it is shown that the family of pointwise suprema of sets of \emph{$c$-elementary} functions, that is, functions $x \in X \en c(x,(x^{\ast},u^{\ast },\alpha))-\beta \in \Ramp$, with $(x^{\ast},u^{\ast },\alpha)\in W$ and $\beta \in \R$, is indeed, the family of proper e-convex functions from $X$ to $\Ramp$ along with the function f $\equiv +\infty$. 

In \cite{FVR2012} it is defined the notion of  \emph{ e$^{\prime}$-convex} function as that function $g:W \en \Ramp$ which is the pointwise supremum of sets of \emph {$c^{\prime}$-elementary} functions, that is, functions $(x^{\ast},u^{\ast },\alpha)\in W \en  c(x,(x^{\ast},u^{\ast },\alpha))-\beta \in \Ramp$ with  $x \in X $ and $\beta \in \R$.  Moreover, the \emph{e$^{\prime}$-convex hull} of any function $g:W \en \Ramp$, ${\epco}g$,  is its largest e$^{\prime}$-convex minorant. The epigraphs of e$^{\prime}$-convex functions are called \emph{e$^{\prime}$-convex} sets, and for any set $D \subset W \times \R$ , its \emph{e$^{\prime}$-convex hull}, denoted by $\epco D$, is the smallest e$^{\prime}$-convex set that contains $D$. We recommend the reader \cite{FV2020} for characterizations of $\ep$-convex sets and additional properties of $\ep$-convex functions.\\

We close this preliminary section with a result that shows the suitability of the $c$-conjugation scheme for e-convex functions and represents the counterpart of Fenchel-Moreau theorem for them.
\begin{theorem}[{\cite[Prop.~6.1,~Prop.~6.2,~Cor.~6.1]{ML2005}}]
\label{th:Theorem2.2}
Let $f:X\en \R\cup \left\{+\infty\right\}$ and $g:W \en \Ramp$ be two functions. Then
\begin{enumerate}
	\item[i)] $f^{c}$ is e$^{\prime}$-convex; $g^{c^{\prime }}$ is e-convex.
	\item[ii)]${\eco}f=f^{c{c^{\prime}}}$ and ${\epco}g=g^{{c^{\prime}}c}$.
	\item[iii)] $f$ is e-convex if and only if $f^{c c^{\prime}}=f$; $g$ is e$^{\prime}$-convex if and only if  $g^{ c^{\prime} c}=g$.
	\item[iv)] $f^{cc^{\prime}} \leq f$; $g^{c^{\prime} c} \leq g$.
\end{enumerate}
\end{theorem}
%

\section{First properties of the $c$-subdifferential }
\label{sec:3}

Given a function, its subdifferentiability at a point associated with the $c$-conjugation scheme was considered in \cite{MLVP2011} as a particularization of the $c$-subdifferentiability introduced first in \cite{BA1977}. 

\begin{definition}[{\cite[Def. 44]{MLVP2011}}]\label{def:c-subgradient}
Let $f:X\en \Ramp$ be a proper function. A vector $ (x^*, u^*, \alpha) \in W$ is a \textit{$c$-subgradient} of $f$ at $x_{0} \in X$ if  $f(x_{0}) \in \R, \langle x_{0}, u^* \rangle < \alpha$ and, for all $x \in X$,
\begin{equation*}
f(x)-f(x_{0}) \geq c(x,(x^*, u^*, \alpha))-c(x_{0}, (x^*, u^*, \alpha)).
\end{equation*}
The set of all $c$-subgradients of $f$ at $x_{0}$ is denoted by $\partial_{c}f(x_{0})$ and is called the \textit{$c$-subdifferential set} of $f$ at $x_0$. In the case $f(x_{0}) \notin \R$, it is set $\partial_{c}f(x_{0})=\emptyset$.
\end{definition}

 In \cite{FGV2021} it was introduced the notion of $c^{\prime}$-subdifferentiability, following \cite{BA1977} too.

\begin{definition}[{\cite[Def. 4.2]{FGV2021}}]\label{def:cprime-subgradient}
Let $g:W\en \Ramp$ be a proper function. Then, $ x \in X$ is a \textit{$c^{\prime}$-subgradient} of $g$ at $(x^*_0,u^*_0,\alpha_0) \in W$ if  $g(x^*_0,u^*_0,\alpha_0)  \in \R, \langle x, u_0^* \rangle < \alpha_0$ and, for all $(x^*,u^*,\alpha) \in W$,
\begin{equation*}
	g(x^*,u^*,\alpha)-g(x^*_0,u^*_0,\alpha_0) \geq c^{\prime}((x^*, u^*, \alpha),x)-c^{\prime}((x^*_0, u^*_0, \alpha_0),x).
\end{equation*}
The set of all $c^{\prime}$-subgradients of $g$ at $(x^*_0,u^*_0,\alpha_0)$ is denoted by $\partial_{c^{\prime}}g(x^*_0,u^*_0,\alpha_0)$ and it is called the \textit{$c^{\prime}$-subdifferential set} of $g$ at $(x^*_0,u^*_0,\alpha_0)$.
 In the case $g(x^*_0,u^*_0,\alpha_0) \notin \R$, it is set $\partial_{c^{\prime}}g(x^*_0,u^*_0,\alpha_0)=\emptyset$.
\end{definition}
The following standard notation will be used throughout the paper. Given fix points $u^* \in X^*$  ($(x,u, \beta) \in X \times X \times \R$, resp.) and $\gamma \in \R$, we denote an open hyperplane in $X$ (in $W$, resp.) by
\begin{equation*}
	H_{u^*, \gamma}^{<}=\left \{ x \in X: \left \langle x, u^* \right \rangle < \gamma \right \}
\end{equation*}
and
\begin{equation*}
	H_{(x, u, \beta), \gamma} ^{<}= \lbrace (x^*, u^*, \alpha) \in W: \langle x, x^* \rangle +  \langle u, u^* \rangle + \alpha \beta <\gamma \rbrace,
\end{equation*}
respectively. 
According to \cite{MLVP2011}, given $f:X \en \Ramp$, its $c$-subdifferential at $x_0 \in \dom f$ can be written as
\begin{equation}\label{id:subd_and_c_subd}
\partial_c f(x_0)=\partial f( x_0) \times \left \{ (u^*, \alpha) \in X^* \times \R: \dom f \subseteq H_{u^*, \alpha}^{<} \right \},
\end{equation}
where $\partial f$ stands for the classical (Fenchel) subdifferential. This relation between the standard subdifferential of $f$, $\partial f$, and its $c$-subdifferential, $\partial _c f$, will play a fundamental role throughout next sections.\\

The following results are proved in \cite{FGV2021}. Lemmas \ref{lemma 5} and \ref{lemma4_5} state the counterparts of \cite[Prop.~5.1, Ch.~I]{ET1976} for $c$-subdifferentials and $c^{\prime}$-subdifferentials, respectively, whereas Proposition \ref{Proposition4_6} extends \cite[Cor.~23.5.1]{R1970} to the $c$-conjugation scheme.

\begin{lemma}[{\cite[Lem. 4.3]{FGV2021}}] \label{lemma 5}
Let $f:X\en \Ramp$ be a proper function and $x_{0} \in \dom f$. Then  $ (x^*,u^*, \alpha) \in  \partial_{c}f(x_{0})$ if and only if $\langle x_{0}, u^* \rangle<\alpha$ and
$f(x_{0}) +f^{c} (x^*,u^*, \alpha)= c(x_{0},(x^*,u^*, \alpha))$.
\end{lemma}
%

\begin{lemma}[{\cite[Lem. 4.4]{FGV2021}}] \label{lemma4_5}
Let $g\colon W\to\Ramp$ be a proper function and $(x^*_0,u^*_0,\alpha_0)\in\dom g$.
Then $x\in\partial_{c^{\prime}}g(x^*_0,u^*_0,\alpha_0)$ if and only if $\ci x,u^*_0\cd <\alpha_0$ and
\begin{equation}\label{eq:Equality_cprime_subdifferential}
	g(x^*_0,u^*_0,\alpha_0) + g^{c^{\prime}}(x) = c^{\prime}((x^*_0,u^*_0,\alpha_0),x).
\end{equation}
\end{lemma}
%

\begin{proposition}[{\cite[Prop.  4.5]{FGV2021}}] \label{Proposition4_6}
Let $f\colon X\to \Ramp$ be a proper function and $x_0 \in \dom f$. If $(x^*,u^*,\alpha) \in \partial_c f(x_0)$, then $ x_0 \in \partial_{c^{\prime}} f^c(x^*,u^*,\alpha)$. The converse statement holds if $f$ is e-convex.
\end{proposition}
%

Now we present the counterparts of some well-known results in classical subdifferential theory. The next theorem sums up some basic properties involving $c$-subdifferential and $c^{\prime}$-subdifferential sets.
\begin{theorem}\label{theorem1}
Let $f\colon X\to \Ramp$  and $g\colon W\to \Ramp$  be  proper functions, $ x_0 \in X$ and $( x_0^*, u_0 ^*, \alpha_0) \in W$. Then:
\begin{enumerate}
	\item[i)] $\partial_c f(x_0) \subseteq W$ and $\partial_{c^{\prime}} g( x_0^*, u_0 ^*, \alpha_0) \subseteq X$ are e-convex sets.
	\item[ii)] If  $\partial_c f(x_0) \neq  \emptyset$, then ${\eco}f(x_0)=f(x_0)$ and, moreover,
	\begin{equation*}
		\partial_c ({\eco}f)(x_0)=\partial_c f(x_0).
	\end{equation*}
	\item[iii)] If $\partial_{c^{\prime}} g( x_0^*, u_0 ^*, \alpha_0) \neq  \emptyset$, then $({\epco}g)( x_0^*, u_0 ^*, \alpha_0) =g( x_0^*, u_0 ^*, \alpha_0)$ and, moreover,
	\begin{equation*}
		\partial_{c^{\prime}} (\epco g)( x_0^*, u_0 ^*, \alpha_0)=  \partial_{c^{\prime}} g( x_0^*, u_0 ^*, \alpha_0).
	\end{equation*}
	\item[iv)] If $x_0 \in \partial_{c^{\prime}} g( x_0^*, u_0 ^*, \alpha_0)$ and $\dom g \subseteq H_{(0, x_0, -1),0} ^{<}$, then $(x_0,0,0) \in \partial g( x_0^*, u_0 ^*, \alpha_0)$.
\end{enumerate}
\end{theorem}
\begin{proof}
$i)$ We can assume that $ x_0\in \dom f$ and  $( x_0^*, u_0 ^*, \alpha_0) \in \dom g$.
From (\ref{id:subd_and_c_subd}) and \cite[Th.~2.4.1]{Z2002}, $ \partial f (x_0)$ is convex and $w^*$-closed in $X^*$, which is a locally convex space, hence $\partial f(x_0)$ is e-convex.\\
Now, name  $V:=\left \{ (u^*, \alpha) \in X^* \times \R: \dom f \subseteq H_{u^*, \alpha}^{<} \right \}$. We will show, in first place, that $V$ is e-convex according to Definition \ref{defeconv}. Take a point $(v^*, \beta) \in X^* \times \R$ not belonging to $V$. Then, there exists $x \in \dom f$ such that $\left \langle x, v^* \right \rangle \geq \beta$. Take $(x, -1) \in X \times \R \subset (X^* \times \R)^*$, and we will have that the hyperplane
$$H=\left \{ (y^*, \gamma) \in X^* \times \R:\left \langle -x, y^* \right \rangle - \gamma = \left \langle -x, v^* \right \rangle - \beta \right \}$$
passes through the point  $(v^*, \beta)$, and for all $(u^*, \alpha) \in V$, 
\begin{equation*}
	\left \langle -x, u^* \right \rangle - \alpha >0 \geq  \left \langle -x, v^* \right \rangle - \beta
\end{equation*}
and $V \cap H = \emptyset$. We obtain that $V$ is e-convex and, recalling (\ref{id:subd_and_c_subd}),  $\partial_c f(\bar x)=\partial f(\bar x) \times V$ is e-convex, too \footnote{ In \cite{RVP2011} it is established that $C \subset \R^n$ and $D \subset \R^m$ are e-convex if and only if $C \times D$ is e-convex, and it can be easily extended to the framework of locally convex spaces.}.\\

Again using Definition \ref{defeconv}, take $x \not \in   \partial_{c^{\prime}} g( x_0^*, u_0 ^*, \alpha_0)  $. Then, either $\left \langle x,  u_0^* \right \rangle \geq  \alpha_0$ or $\left \langle x, u_0^* \right \rangle < \alpha_0$ and a point $( x^*,  u ^*, \alpha) \in W$ can be found verifying 
\begin{equation}\label{eq3}
g(x^*,u^*,\alpha)<g( x_0^*, u_0 ^*, \alpha_0) +c^{\prime}((x^*, u^*, \alpha),x)- c^{\prime}(( x_0^*, u_0 ^*, \alpha_0),x).
\end{equation}
In the first case, take $(u_0 ^*,  \alpha_0) \in X^* \times \R$, and we have 
$$\left \langle y,  u_0^* \right \rangle < \alpha_0,$$
for all $y \in  \partial_{c^{\prime}} g( x_0^*, u_0 ^*, \alpha_0) $. The hyperplane $\left \{ z \in X: \left \langle z,  u_0^* \right \rangle =\left \langle x,  u_0^* \right \rangle \right \}$ verifies that it contains the point $x$ and has empty intersection with  $\partial_{c^{\prime}} g( x_0^*, u_0 ^*, \alpha_0) $, allowing us to conclude that this set is e-convex.\\
In the second case, if moreover $\left \langle x,  u^* \right \rangle <  \alpha$, we can take $(x^*-x_0^*, \beta_0) \in X^* \times \R$, where
$$ \beta_0=g(x^*,u^*,\alpha)-g( x_0^*, u_0 ^*, \alpha_0) \in \R.$$
Since (\ref{eq3}) holds, $ \left \langle x,  x^*- x_0^* \right \rangle >\beta_0$ and, for all $y \in  \partial_{c^{\prime}} g( x_0^*, u_0 ^*, \alpha_0) $,
\begin{equation*}
	\left \langle y,  x^*- x_0^* \right \rangle \leq \beta_0.
\end{equation*}
Naming $ \delta_0=\left \langle x,  x^*- x_0^*  \right \rangle$ and taking the hyperplane $\left \{ z \in X: \left \langle z,  x^*- x_0^*  \right \rangle = \delta_0\right \}$, we have
\begin{equation*}
	\left \langle y,  x^*- x_0^* \right \rangle \leq \beta_0< \delta_0,
\end{equation*}
for all  $y \in  \partial_{c^{\prime}} g( x_0^*, u_0 ^*, \alpha_0) $, and $\left \langle x,  x^*- x_0^*  \right \rangle=\delta_0.$ 
Finally, still in the second case, but $\left \langle x,  u^* \right \rangle \geq \alpha$, take  $( u^*, \alpha) \in X^* \times \R$, which verifies that $\left \langle y,  u^* \right \rangle < \alpha$, for all $y \in  \partial_{c^{\prime}} g( x_0^*, u_0 ^*, \alpha_0) $. We conclude that also in the second case  $\partial_{c^{\prime}} g( x_0^*, u_0 ^*, \alpha_0)$ is e-convex.\\

$ii)$ We set $\varphi: X \to\Ramp$ the following $c$-elementary and, consequently, e-convex function
\begin{equation*}
	\varphi (x)=c(x,(x^*,u^*,\alpha))-c( x_0,(x^*,u^*,\alpha)) + f( x_0),
\end{equation*}
where $(x^*,u^*,\alpha) \in \partial_c f(x_0)$.
Then, for all $x \in X$, $\varphi (x) \leq	 f(x)$, and since $\eco f$ is the largest e-convex minorant of $f$ by definition, we have
$$\varphi (x) \leq \eco f(x) \leq f(x), \text{ for all } x \in X,$$
but $\varphi (x_0)= f( x_0)$, hence  ${\eco}f( x_0)=f( x_0)$.
Now, according to Lemma \ref{lemma 5}, $(x^*,u^*,\alpha) \in \partial_c f(x_0)$ if and only if $\left \langle x_0, u^* \right \rangle < \alpha$ and 
\begin{equation*}
	f(x_0)+f^{c} (x^*,u^*, \alpha)= c( x_0,(x^*,u^*, \alpha)).
\end{equation*}
On the other hand, as it is proved in \cite[Prop.~40]{MLVP2011}, $\eco f= f^{cc^{\prime}}$ and then, ${(\eco f)^c}=(f^{cc^{\prime}})^c=(f^c)^{{c^{\prime}}c}=f^c$, due to the fact
that $f^c$ is  e$^{\prime}$-convex and Theorem \ref{th:Theorem2.2} is applied. Hence the above equality reads now
\begin{equation*}
	\eco f( x_0)+(\eco f)^{c} (x^*,u^*, \alpha)= c(x_0,(x^*,u^*, \alpha)),
\end{equation*}
and applying again Lemma \ref{lemma 5}, we have $(x^*,u^*,\alpha) \in \partial_c (\eco f)(x_0)$. Therefore $\partial_c f(x_0) \subseteq \partial_c ({\eco}f)(x_0)$.
The opposite inclusion follows in an analogous way.\\

$iii)$ We set, for each $x \in \partial_{c^{\prime}} g( x_0^*, u_0 ^*, \alpha_0)$, the $c^{\prime}$-elementary and therefore $e^{\prime}$-convex function $\varphi_x^{\prime}\colon W \to \Ramp$,
\begin{equation*}
	\varphi_x^{\prime}( x^*,  u ^*,  \alpha)=c^{\prime}((x^*, u^*, \alpha),x)- c^{\prime}(( x_0^*, u_0 ^*, \alpha_0),x)+g( x_0^*, u_0 ^*, \alpha_0).
\end{equation*}
Then, for all $( x^*,  u ^*,  \alpha) \in W$, $\varphi_x^{\prime}( x^*,  u ^*,  \alpha) \leq g( x^*,  u ^*, \alpha)$, and 
\begin{equation*}
	\varphi_x^{\prime}( x^*,  u ^*,  \alpha) \leq \epco g( x^*,  u ^*,  \alpha) \leq g( x^*,  u ^*, \alpha),
\end{equation*}
because $\varphi_x^{\prime}$ is $e^{\prime}$-convex, but $\varphi_x^{\prime}( x_0^*, u_0 ^*, \alpha_0)=g( x_0^*, u_0 ^*, \alpha_0)$, so $\epco g( x_0^*, u_0 ^*, \alpha_0)= g( x_0^*, u_0 ^*, \alpha_0)$. The second equality follows the same steps than the case of $c$-subdifferentiability in item \emph{ii)}.\\

$iv)$ If $ x_0 \in \partial_{c^{\prime}} g( x_0^*, u_0 ^*, \alpha_0)$, we have  $\left \langle x_0, \ u_0^* \right \rangle < \alpha_0$ and 
\begin{equation}\label{eq2}
	g(x^*,u^*,\alpha) \geq g( x_0^*, u_0 ^*, \alpha_0) +c^{\prime}((x^*, u^*, \alpha),x_0)-\left \langle x_0, x_0^* \right \rangle,
\end{equation}
for all $(x^*,u^*,\alpha) \in \dom g$. Since $\dom g \subseteq H_{(0, x_0, -1)} ^{<}$,  for all $(x^*,u^*,\alpha) \in \dom g$, $\left \langle x_0,  u^* \right \rangle < \alpha$,  then we can rewrite (\ref{eq2}) as
\begin{equation*}
	g(x^*,u^*,\alpha) \geq g( x_0^*, u_0 ^*, \alpha_0) +\left \langle x_0, x^* \right \rangle-\left \langle x_0,  x_0^* \right \rangle,
\end{equation*}
and $( x_0,0,0) \in \partial g( x_0^*, u_0 ^*, \alpha_0)$.
\end{proof}

%
\begin{remark}
Statement $iv)$ in the previous theorem  establishes a relationship between subdifferential and $c^{\prime}$-subdifferential sets, as it occurs in the case of $c$-subdifferentiability; recall  (\ref{id:subd_and_c_subd}). 
\end{remark}

The notion of $\varepsilon$-$c$-subgradient appears firstly in \cite{FVR2012} allowing a characterization of the epigraph of the $c$-conjugate (see \cite [Lem.~9]{FVR2012}). Moreover, it is used to build an alternative formulation for a general regularity condition in evenly convex optimization in \cite [Th.~4.9] {FGV2021}.
\begin{definition}[{\cite[Def. 4]{FVR2012}}]
Let $f:X\en \Ramp$ be a proper function and $\varepsilon \geq 0$. A vector $ (x^*, u^*, \alpha) \in W$ is a \textit{$\varepsilon$-c-subgradient} of $f$ at $x_{0} \in X$ if  $f(x_{0}) \in \R, \langle x_{0}, u^* \rangle < \alpha$ and, for all $x \in X$,
\begin{equation*}
f(x)-f(x_{0}) \geq c(x,(x^*, u^*, \alpha))-c(x_{0}, (x^*, u^*, \alpha))-\varepsilon.
\end{equation*}
The set of all $\varepsilon$-$c$-subgradients of $f$ at $x_{0}$ is denoted by $\partial_{c,\varepsilon}f(x_{0})$ and is called the \textit{$\varepsilon$-$c$-subdifferential set} of $f$ at $x_0$. In the case $f(x_{0}) \notin \R$, it is set $\partial_{c,\varepsilon}f(x_{0})=\emptyset$.
\end{definition}
Relating the notion of $\varepsilon$-$c$-subdifferentiability with the classical $\varepsilon$-subdifferentiability, it is easy to see that, for all $x_0 \in \dom f$, 
\begin{equation}\label{relationepsilon}
\partial_{c,\varepsilon}f(x_{0})=\partial_{\varepsilon} f(x_0) \times \left \{ (u^*, \alpha) \in X^* \times \R: \dom f \subset H_{u^*, \alpha}^{<} \right \}.
\end{equation}
Moreover, we also have that, if $0 \leq \varepsilon_1 \leq \varepsilon_2 < \infty$, then 
\begin{equation*}
	\partial_{c}f(x_{0})=\partial_{c,0}f(x_{0}) \subseteq \partial_{c,\varepsilon_1}f(x_{0}) \subseteq \partial_{c,\varepsilon_2}f(x_{0}),
\end{equation*}
and
\begin{equation*}
	\partial_{c,\varepsilon}f(x_{0})=\bigcap_{\eta > \varepsilon} \partial_{c,\eta}f(x_{0}).
\end{equation*}

In a similar way, we can define the \textit{$\varepsilon$-$c^{\prime}$-subdifferential set} of a function $g\colon W\to \Ramp$ at $(x^*_0,u^*_0,\alpha_0) \in W$.

\begin{theorem}\label{theorem3}
Let $f,g\colon X\to \Ramp$  be proper functions, $ x_0 \in X$ and  $\varepsilon \geq 0$. Then:
\begin{enumerate}
	\item[i)]  $\partial_{c, \varepsilon}f(x_0) \subseteq W$ is e-convex.
	\item[ii)] $( x^*,  u ^*, \alpha) \in \partial_{c, \varepsilon} f(x_0)$ if and only if $\langle x_0,  u^* \rangle < \alpha$ and 
	\begin{equation*}
		f( x_0)+f^c( x^*,  u ^*,  \alpha) \leq c( x_0, ( x^*,  u ^*,  \alpha) )+\varepsilon.
	\end{equation*}
	Moreover, $ x_0 \in   \partial_{c^{\prime}, \varepsilon} f^c( x^*,  u ^*,  \alpha)$.
	\item[iii)] If $x_0 \in \dom f \cap \dom g$, then 
	\begin{equation*}
		\bigcup_{\eta \in [0, \varepsilon]} \left (\ \partial_{c, \eta} f(x_0)+\partial_{c, \varepsilon-\eta} g( x_0) \right )\ \subseteq \partial_{c, \varepsilon} (f+g)(x_0).
	\end{equation*}
\end{enumerate}
\end{theorem}
%
\begin{proof}
$i)$ Similar to the proof of item $i)$ in  Theorem \ref{theorem1}.\\

$ii)$ $( x^*,  u ^*, \alpha) \in \partial_{c, \varepsilon} f( x_0)$ if and only if $ \langle x_0,  u^* \rangle <  \alpha \text{ and, for all } x \in X$,
\begin{equation*}
	f(x_0)+ c(x,( x^*,  u ^*, \alpha))- f(x) \leq c( x_0, ( x^*,  u ^*, \alpha))+\varepsilon,
\end{equation*}
equivalently, $\langle x_0, u^* \rangle < \alpha$ and
\begin{equation*}
	f(x_0)+f^c ( x^*, u^*, \alpha) \leq c( x_0, ( x^*, u ^*, \alpha))+\varepsilon.
\end{equation*}
Analogously, $x_0 \in   \partial_{c^{\prime}, \varepsilon} f^c( x^*, u ^*, \alpha)$ if and only  if 
\begin{equation*}
	\langle  x_0,  u^* \rangle <  \alpha \text{ and }  f^c ( x^*,  u^*,  \alpha) +f^{cc^{\prime}} (x_0) \leq c^{\prime}( ( x^*, u ^*, \alpha),x_0)+\varepsilon,
\end{equation*}
and taking into account that, according to Theorem \ref{th:Theorem2.2}, $f( x_0) \geq f^{cc^{\prime}} ( x_0)$, we obtain that, if $( x^*,  u^*,  \alpha) \in \partial_{c, \varepsilon} f(x_0)$, then  $ x_0 \in   \partial_{c^{\prime}, \varepsilon} f^c (x^*,  u^*,  \alpha).$\\

$iii)$ Let $\eta \in [0, \varepsilon]$. Then $( x_1 ^*, u_1 ^*, \alpha_1 ) \in \partial_{c, \eta} f(x_0)$ and $(x_2 ^*, u_2 ^*, \alpha_2 ) \in \partial_{c, \varepsilon-\eta} g(x_0)$ if and only if $  \langle x_0,  u_1^* \rangle < \alpha_1$,  $\langle x_0, u_2^* \rangle < \alpha_2$,
\begin{align*}
	f( x_0)+f^c ( x_1 ^*, u_1 ^*, \alpha_1 )&\leq c( x_0,( x_1 ^*, u_1 ^*, \alpha_1 ))+\eta, \text{ and }\\[0.15cm]
	g( x_0)+g^c (x_2 ^*, u_2 ^*, \alpha_2 )  &\leq c( x_0, (x_2 ^*, u_2 ^*, \alpha_2 ))+\varepsilon -\eta.
\end{align*}
Then  $\langle  x_0, u_1^* + u_2^* \rangle < \alpha_1 + \alpha_2$ and,  adding both inequalities and considering the additivity property of the coupling function, which can be applied in this case, we obtain
\begin{equation*}
	(f+g)( x_0)+f^c ( x_1 ^*, u_1 ^*, \alpha_1 )+g^c (x_2 ^*, u_2 ^*, \alpha_2 )  \leq c( x_0, (x_1 ^*+ x_2^*,  u_1 ^*+  u_2 ^*, \alpha_1 +  \alpha_2 ))+\varepsilon.
\end{equation*}
However, it yields
\begin{align*}
	f^c ( x_1 ^*, u_1 ^*, \alpha_1 ) &+ g^c (x_2 ^*, u_2 ^*, \alpha_2 ) \\[0.15cm]
	&\geq \sup_{x \in X} \left \{ c(x, ( x_1 ^*, u_1 ^*, \alpha_1 ))+ c(x, (x_2 ^*, u_2 ^*, \alpha_2 ) )-(f+g) (x) \right \}\\[0.15cm]
	&\geq \sup_{x \in X} \left \{c(x, (x_1 ^*+ x_2^*,   u_1 ^*+ u_2 ^*, \alpha_1 + \alpha_2 ))-(f+g)( x) \right \}\\[0.15cm]
	&= (f+g)^c( x_1 ^*+  x_2^*,  u_1 ^*+  u_2 ^*, \alpha_1 +  \alpha_2 ),
\end{align*}
thus, we conclude
\begin{align*}
(f+g)( x_0)&+(f+g)^c(x_1 ^*+ x_2^*,  u_1 ^*+  u_2 ^*,  \alpha_1 + \alpha_2 )\\[0.15cm]
 &\leq c(x_0, ( x_1 ^*+ x_2^*,  u_1 ^*+ u_2 ^*,  \alpha_1 + \alpha_2 ))+\varepsilon,
\end{align*}
and, consequently, $(x_1 ^*+ x_2^*, u_1 ^*+  u_2 ^*, \alpha_1 + \alpha_2 ) \in \partial_{c, \varepsilon} (f+g)(x_0)$.
\end{proof}

\section{Further properties of the  $c$-subdifferential}
\label{sec:4}

We continue developing additional properties of the $c$-subdifferential extending them from \cite{Z2002}. Next proposition establishes the relationship between the $c$-subdifferential of a proper function and the domain of its $c$-conjugate.
\begin{proposition}\label{prop:domfc_vs_csubdiff} 
Let $f:X\to\Ramp$ be a proper function and $x_0\in \dom f$. Then $\partial_c f(x_0)\subseteq \dom f^c$.
\end{proposition}
%
\begin{proof}
Take $(x^*,u^*,\alpha)\in\partial_c f(x_0)$. Then, by definition it holds
\begin{equation*}
	f(x)-f(x_0)\geq c(x,(x^*,u^*,\alpha)) - c(x_0,(x^*,u^*,\alpha)),~ \forall x\in X,
\end{equation*}
together with $\ci x_0,u^*\cd <\alpha$. Rearranging, this means that
\begin{equation*}
	c(x_0,(x^*,u^*,\alpha)) - f(x_0) \geq c(x,(x^*,u^*,\alpha))-f(x),~\forall x\in X,
\end{equation*}
with $\ci x_0,u^*\cd<\alpha$, which, after taking supremum on the right-hand-side, can be rewritten as
\begin{equation*}
	c(x_0,(x^*,u^*,\alpha))-f(x_0) \geq f^c(x^*,u^*,\alpha).
\end{equation*}
Since $\ci x_0,u^*\cd<\alpha$, the coupling function is finite and due to the fact that $x_0\in\dom f$ by hypothesis, we conclude that $(x^*,u^*,\alpha)\in\dom f^c$.\\
\end{proof}

%
\begin{remark}
In a similar way, it can be proved that $\partial_{c^{\prime}} g( x^*, u ^*, \alpha) \subseteq \dom g^{c^{\prime}}$, for a proper function $g\colon W\to \Ramp$ and $( x^*,  u ^*,\alpha) \in \dom g$.
\end{remark}
According to Lemma \ref{lemma 5}, for any proper function $f$ it holds that a point $(x^*,u^*,\alpha) \in \dom f^c$ belongs to $\partial_c f(x_0)$ if and only if $f^c(x^*,u^*,\alpha)=c(x_0,(x^*,u^*,\alpha))-f(x_0)$, which means that $\sup_X\left \{ c(x,(x^*,u^*,\alpha))-f(x) \right \}$ must be attained at $x_0$. This is not necessarily true, so the inclusion in Proposition \ref{prop:domfc_vs_csubdiff} may be strict.
\begin{example}
Let $f \colon \R \to \Ramp$,
\begin{equation*} 
	f(x) =\left\{
	\begin{array}{ll}
		x^2,&  \text{ if } x >0,\\[0.2cm]
		+\infty,&  \text{ otherwise.}%
	\end{array}
	\right.
\end{equation*}
Let us observe that $(3,0,1) \in \dom f^c$, since $f^c(3,0,1)=\sup_{x>0}\{3x-x^2\}={9 \over 4}$, and it is attained at $x={3 \over 2}$, hence $(3,0,1) \notin  \partial_c f(x_0)$ for all $x_0 \neq {3 \over 2}$. Anyway, a matter of computation shows
$$\dom f^c=\R \times  (\ \R_{-} \times \R_{+} \backslash \{ (0,0) \}),$$
whereas, for all $x_0>0$,
$$\partial_c f(x_0)=\{2x_0\} \times (\ \R_{-} \times \R_{+} \backslash \{ (0,0) \}).$$
\end{example}

 For any proper function $f:X\to\Ramp$, any point $x_0\in\dom f$   and $(x^*,u^*,\alpha)\in\partial_c f(x_0)$, we can infer that, since for all $x\in X$,
\begin{equation*}
	f(x)\geq f(x_0) +c(x,(x^*,u^*,\alpha)) - c(x_0,(x^*,u^*,\alpha)),
\end{equation*}
then,
\begin{equation*}
	f(x)\geq f(x_0) +\sup_{(x^*,u^*,\alpha)\in\partial_c f(x_0)} \left\{ c(x,(x^*,u^*,\alpha)) - c(x_0,(x^*,u^*,\alpha))\right\},
\end{equation*}
 for all $x \in X$.
Next proposition shows that equality between  $\dom f^c$ and $\partial_c f(x_0)$ is a sufficient condition for the above inequality to be an equality, if $f$ is e-convex.

\begin{proposition}\label{equality_f_fc_and_c's} Let $f:X\to\Ramp$ be a proper e-convex function with $x_0\in\dom f$ and $\partial_c f(x_0)= \dom f^c$. Then, for all $x \in X$,
\begin{equation*}
	f(x)= f(x_0) +\sup_{(x^*,u^*,\alpha)\in\partial_c f(x_0)} \left\{ c(x,(x^*,u^*,\alpha)) - c(x_0,(x^*,u^*,\alpha))\right\}.
\end{equation*}
\end{proposition}
%
\begin{proof}
Due to Lemma \ref{lemma 5}, for any $(x^*,u^*,\alpha)\in\partial_c f(x_0)$ we have that
\begin{equation*}
	 f(x_0)+f^c(x^*,u^*,\alpha) = c(x_0,(x^*,u^*,\alpha)),
\end{equation*}
so multiplying by $-1$ and adding $c(x,(x^*,u^*,\alpha))$ we get
\begin{equation*}
	-f(x_0)+c(x,(x^*,u^*,\alpha)) -f^c(x^*,u^*,\alpha) = c(x,(x^*,u^*,\alpha)) - c(x_0,(x^*,u^*,\alpha)),
\end{equation*}
being this equality true for all $x\in X$ and $(x^*,u^*,\alpha)\in\partial_c f(x_0)$. Taking supremum over $\dom f^c=\partial_c f(x_0)$ in both sides of the previous equality, we have, for all $x \in X$,
\begin{equation*}
	-f(x_0) + f^{cc^\prime}(x) =\sup_{(x^*,u^*,\alpha)\in\partial_c f(x_0)} \left\{ c(x,(x^*,u^*,\alpha)) - c(x_0,(x^*,u^*,\alpha)) \right\}.
\end{equation*}
Since $f$ is e-convex by hypothesis, according to Theorem \ref{th:Theorem2.2}, $f=f^{cc^\prime}$, and the proof ends.
\end{proof}

The goal now is to relate the $\varepsilon$-directional derivative with the $\varepsilon$-$c$-subdifferential of $f$. We recall the definition of the $\varepsilon$-directional derivative of a function $f$ at a point $x$ along the direction $u$, i.e., 
\begin{equation*}
	f_\varepsilon^\prime(x,u) := \inf_{t>0} \frac{f(x+tu)-f(x)+\varepsilon}{t},
\end{equation*}
see \cite[Th.~2.1.14]{Z2002}. Observe that \cite[Th.~2.4.4]{Z2002} states the relationship between $\partial f_\varepsilon^\prime (x_0,\cdot)(0) $ and $\partial_\varepsilon f(x_0)$ for a proper function and $ x_0 \in \dom f$. The following theorem deals with this relation when $c$-subdifferentiability is used. First, we recall the definition of the normal cone of a convex set $C\subset X$ at a point $ x_0 \in C$, 
\begin{equation*}
	\mathcal{N} (C,x_0) = \left\{ u^*\in X^*\,:\, \ci x-x_0,u^*\cd \leq 0,~\text{ for all } x\in C\right\}.
\end{equation*}
\begin{theorem}\label{th:Equality_subdif_f_and_ddf} 
Let $f:X\to\Ramp$ be a proper function, $x_0\in\dom f$ and $\varepsilon \geq 0$. Then
\begin{equation*}
	\partial_cf_\varepsilon^\prime (x_0,\cdot)(0)\cap (X^*\times\left\{ (u^*,\alpha)\,:\,\ci x_0, u^*\cd<\alpha\right\}) = \partial_{c,\varepsilon}f(x_0)\cap (X^*\times \mathcal{N}(\dom f,x_0)\times \R_{++}).
\end{equation*}
\end{theorem}
%
\begin{proof}
Take $(x^*,u^*,\alpha)\in\partial_cf_\varepsilon^\prime(x_0,\cdot)(0)$ such that $\ci x_0, u^*\cd<\alpha$. Then $\alpha>0$ and
\begin{equation*}
	f_\varepsilon^\prime(x_0,u) - f_\varepsilon^\prime(x_0,0) \geq c(u,(x^*,u^*,\alpha)),~\text{ for all }u\in X.
\end{equation*}
By the definition of $f_\varepsilon^\prime$,
\begin{equation*}
	\frac{f(x_0+tu)-f(x_0)+\varepsilon}{t} \geq c(u,(x^*,u^*,\alpha)),~\text{for all } u\in X,t>0.
\end{equation*}
Let $x:=x_0+tu$. Hence, for all $x\in X$ and $t>0$ it holds
\begin{equation}\label{eq:1_lemma_equality_subdif_f_and_ddf}
	\frac{f(x)-f(x_0)+\varepsilon}{t} \geq c\left(\frac{x-x_0}{t},(x^*,u^*,\alpha)\right).
\end{equation}
In particular, for $t=1$, we get
\begin{equation*}
	f(x)-f(x_0)\geq c(x-x_0,(x^*,u^*,\alpha))-\varepsilon,~\text{ for all } x\in X,
\end{equation*}
then, due to the subadditivity of the coupling function
\begin{equation*}
	f(x)-f(x_0)\geq c(x,(x^*,u^*,\alpha)) - c(x_0,(x^*,u^*,\alpha))-\varepsilon,~\text{ for all } x\in X.
\end{equation*}
Since $x_0\in\dom f$ and $\ci x_0,u^*\cd<\alpha$, we have that $(x^*,u^*,\alpha)\in\partial_{c,\varepsilon} f(x_0)$. Having in mind that $\alpha>0$, from (\ref{eq:1_lemma_equality_subdif_f_and_ddf}), we have, for all $x\in\dom f$ and $t>0$
\begin{equation*}
	\frac{x-x_0}{t}\in H_{u^*,\alpha}^<,
\end{equation*} 
so $\ci x-x_0,u^*\cd< \alpha t$ for all $t>0$, which means that $\ci x-x_0,u^*\cd\leq 0$ for all $x\in\dom f$ and we get that $u^*\in \mathcal{N}(\dom f,x_0)$.\\

For the reverse inclusion, take $(x^*,u^*,\alpha) \in\partial_{c,\varepsilon} f(x_0)$ with $\alpha>0$ and $u^*\in \mathcal{N}(\dom f, x_0)$. We will prove that
\begin{equation*}
	f_\varepsilon^\prime(x_0,u) -f_\varepsilon^\prime(x_0,0) \geq c(u,(x^*,u^*,\alpha)) - c(0,(x^*,u^*,\alpha)),
\end{equation*}
for all $u\in X$, or, equivalently,
\begin{equation*}
	f_\varepsilon^\prime(x_0,u) \geq c(u,(x^*,u^*,\alpha)).
\end{equation*}
By hypothesis, 
$\ci x_0,u^*\cd<\alpha$ and
\begin{equation*}
	f(x)-f(x_0)\geq c(x,(x^*,u^*,\alpha)) - c(x_0,(x^*,u^*,\alpha))-\varepsilon,~\text{ for all } x\in X.
\end{equation*}
Take $x=x_0+tu$, for all $u \in X$ and $t>0$, then
\begin{equation*}
	f(x_0+tu)-f(x_0)\geq c(x_0+tu,(x^*,u^*,\alpha))-c(x_0,(x^*,u^*,\alpha))-\varepsilon,
\end{equation*}
which gives us
\begin{equation*}
	\frac{f(x_0+tu)-f(x_0)+\varepsilon}{t} \geq \frac{1}{t}c(x_0+tu,(x^*,u^*,\alpha)) - \frac{1}{t}c(x_0,(x^*,u^*,\alpha)) \geq \frac{1}{t}c(tu,(x^*,u^*,\alpha)),
\end{equation*}
for all $u \in X$ and $t>0$. Hence, for each $u\in X$,
\begin{equation*}
	\inf_{t>0} \frac{f(x_0+tu)-f(x_0)+\varepsilon}{t} \geq \sup_{t>0} h(t),
\end{equation*}
where
\begin{equation*}
	h(t)=-\frac{1}{t}c(-tu,(x^*,u^*,\alpha))
	=
	\left\{
	\begin{aligned}
		\ci u,x^*\cd & ~ \text{ if } \ci -tu,u^*\cd <\alpha,\\
		-\infty & ~ \text{ otherwise}.
	\end{aligned}
	\right.
\end{equation*}
In case $\ci u,u^*\cd \geq 0$, for all $t>0$ we will have $\ci -tu,u^*\cd\leq 0<\alpha$. If $\ci u,u^*\cd<0$, we can take $t$ small enough, since $t\downarrow 0^+$, such that $\ci -tu,u^*\cd<\alpha$. Hence,
\begin{equation*}
	\sup_{t>0} h(t) = \ci u,x^*\cd 
\end{equation*}
and
\begin{equation}\label{eq:2_lemma_equality_subdif_f_and_ddf}
	f_\varepsilon^\prime(x_0,u)\geq \ci u,x^*\cd. 
\end{equation}
Taking into account that $u^*\in\mathcal{N}(\dom f,x_0)$, for all $y\in\dom f$ it holds $\ci y-x_0,u^*\cd\leq 0<\alpha$. Then, for all $t>0$ we get $\ci t(y-x_0),u^*\cd<\alpha$ and
\begin{equation*}
	\cone(\dom f-x_0)\subset H_{u^*,\alpha}^<
\end{equation*}
which, according to \cite [Th.2.1.14]{Z2002}, means that $\dom f_\varepsilon^\prime(x_0,\cdot)\subset H_{u^*,\alpha}^<$ and, hence, in (\ref{eq:2_lemma_equality_subdif_f_and_ddf}) we can write
\begin{equation*}
	f_\varepsilon^\prime(x_0,u)\geq c(u,(x^*,u^*,\alpha)), ~\text{ for all } u\in X,
\end{equation*}
concluding the proof.
\end{proof}

Due to Theorem \ref{th:Equality_subdif_f_and_ddf}, we pursue the counterpart of \cite[Th.~2.4.11]{Z2002}, which establishes that if $f$ is proper convex and lsc, $x_0\in\dom f$ and $\varepsilon>0$, then
\begin{equation*}
	f_\varepsilon^\prime(x_0,u) = \sup\left\{ \ci u,x^*\cd\,:\,u^*\in\partial_\varepsilon f(x_0)\right\}, ~\text{for all } u\in X.
\end{equation*}
Next result comes immediately from Theorem \ref{th:Equality_subdif_f_and_ddf}.
\begin{corollary}\label{cor:Equality_subdif_f_and_ddf}
Let $f:X\to\Ramp$ be a proper function with $x_0\in\dom f$ and $\varepsilon>0$. Then, for all $u\in X$ it holds
\begin{equation*}
	f_\varepsilon^\prime(x_0,u) \geq \sup\left\{ c(u,(x^*,u^*,\alpha))\,:\,(x^*,u^*,\alpha)\in\partial_{c,\varepsilon} f(x_0) \cap (X^*\times \mathcal{N}(\dom f,x_0)\times \R_{++})\right\}. 
\end{equation*}
\end{corollary}
\begin{remark}
It is not an easy task to find sufficient conditions for Corollary \ref{cor:Equality_subdif_f_and_ddf} to hold with an equality, maybe a separation theorem for e-convex sets needs to be studied, in some way, besides asking the function $f$ to be e-convex, for instance. Hence, we have decided to leave this problem to work on it in a near future.
\end{remark}

\section{Optimality conditions via $c$-subdifferentials}
\label{sec:5}

 In \cite{HU1988}, Hiriart-Urruty established $\varepsilon$-optimality and global optimality conditions for DC programs, which are, recall, optimization problems of the type $\inf_X \{ f(x)-g(x) \}$, where $f$ and $g$ are convex functions. To avoid ambiguity, we will use the usual convention $+ \infty-(+\infty)=+\infty$ when minimizing DC problems. Recall that, for any $\varepsilon \geq 0$, a point $a \in X$ is said to be an  $\varepsilon$-minimizer of a function $h:X \to \Ramp$ if $a \in \dom h $ and $$h(a)-\varepsilon \leq h(x),$$ for all $x \in X$.
Those optimality conditions are obtained via subdifferential and $\varepsilon$-subdifferential sets of the involved functions in the problem. \\

\begin{theorem}[{\cite[Th. 4.4]{HU1988}}]\label{lemmaHU}
Let $f,g \colon X \to \Ramp$ be proper convex and lsc functions. A necessary and sufficient condition for the point $a$ to be an $\varepsilon$-minimizer of $f-g$ is that  
$$\partial_{\lambda} g(a) \subseteq \partial_{\varepsilon+\lambda} f(a), \text{ for all } \lambda \geq 0.$$
In particular, $a \in X$ is a global minimizer of $f-g$ if and only if
$$\partial_\varepsilon g(a) \subseteq \partial_\varepsilon f(a), \text{ for all } \varepsilon \geq 0.$$
\end{theorem}
\begin{remark}
An alternative proof of Theorem \ref{lemmaHU} is given in \cite{MLS1992}, where it is pointed out that the convexity and lower semicontinuity of $f$ are not essential assumptions.
\end{remark}
In this section, our purpose is to provide a counterpart of Theorem \ref{lemmaHU}, expressed in terms of even convexity and $c$-subdifferentiability. We will use the following definition of the difference of two sets in $W$.%


\begin{definition}[{\cite{HU1986}}]\label{def:Minus_dot_operation}
Given $A, B \subseteq W$, the \emph{star-difference} between $A$ and $B$ is
\begin{equation*}
	A \overset{*}{-}B=\left \{ (x^*, u^*, \alpha) \in W: \lbrace (x^*, u^*, \alpha) \rbrace+B \subset A \right \},
\end{equation*}
involving the Minkowski addition of two sets.
\end{definition}
\begin{remark}
In case $B=\emptyset$, for all $A \subseteq W$, $A\overset{*}{-}B=A$. Moreover, the difference $\emptyset \overset{*}{-}B$  only has meaning (and it is $\emptyset$) whenever $B=\emptyset$.
\end{remark}
Next lemma can be derived from \cite[Th.~3.1]{ML1990}, which is stated in the generalized conjugation theory framework. We include the proof for the sake of completeness.
\begin{lemma}\label{lemma4}
Let $f \colon X \to \Ramp$ be a proper function and $g \colon X \to \Ramp$ be a proper convex lsc function. Then
\begin{equation*}
	\sup_{x\in X} \left \{ g(x)-f(x) \right \}=\sup_{(x^*, u^*, \alpha) \in W} \left \{ f^c (x^*, u^*, \alpha)-g^c(x^*, u^*, \alpha) \right \}.
\end{equation*}
\end{lemma}
\begin{proof}
From Theorem \ref{th:Theorem2.2}, it yields
\begin{equation*}
	\sup_{x\in X} \left\{ g(x)-f(x) \right\} = \sup_{x\in X} \left\{ g^{cc^{\prime}} (x)-f(x) \right\}
\end{equation*}
and applying the definition of $c^\prime$-conjugate,
\begin{align*}
	\sup_{x\in X} \left\{ g(x)-f(x) \right\} &= \sup_{x\in X} \left\{ \sup_{(x^*, u^*, \alpha) \in W} \left\{ c^{\prime}( (x^*, u^*, \alpha),x)-g^c (x^*, u^*,\alpha)\right\} -f(x) \right \}\\[0.2cm]
 	&= \sup_{(x^*, u^*, \alpha) \in W}\left \{ \sup_{x\in X} \left \{ c(x,(x^*, u^*, \alpha))- f(x) \right \} - g^c (x^*, u^*, \alpha)\right \}\\[0.2cm]
 	&= \sup_{(x^*, u^*, \alpha) \in W} \left \{ f^c (x^*, u^*, \alpha)-g^c(x^*, u^*, \alpha) \right \},
\end{align*}
where in the last equality we have used the definition of $c$-conjugate function.
\end{proof}

The following theorem is stated with equality when $f$ and $g$ are as in Theorem \ref{lemmaHU} and $\varepsilon$-subdifferential sets are used; see Theorem 1 in \cite{MLS1992}.
\begin{theorem} \label{lemma5}
Let $f,g \colon X \to \Ramp$ be proper functions with $g$ e-convex, and $\varepsilon \geq 0$. Then, for all $x \in X$, it holds
\begin{equation*}
	\partial_{c, \varepsilon} (f-g) (x) \subseteq \bigcap_{\lambda \geq 0} \left \{ \partial_{c, \varepsilon + \lambda} f(x) \overset{*}{-}  \partial_{c, \lambda} g(x) \right \}.
\end{equation*}

\end{theorem}
\begin{proof}
Assume that $\partial_{c, \varepsilon} (f-g) (x_0) \neq \emptyset$, for some $ x_0 \in X$, and take any $(x^*, u^*, \alpha) \in \partial_{c, \varepsilon} (f-g) ( x_0)$ and $\lambda \geq 0$. We will show the inclusion
\begin{equation}\label{eq6}
(x^*, u^*, \alpha) + \partial_{c, \lambda} g(x_0) \subseteq \partial_{c, \varepsilon + \lambda} f(x_0).
\end{equation}
By Theorem \ref{theorem3} $ii)$, $(x^*, u^*,\alpha) \in \partial_{c, \varepsilon} (f-g) ( x_0)$ if and only if  $ \langle   x_0,  u^* \rangle <  \alpha$ and
\begin{equation}\label{eq7}
(f-g)( x_0) + (f-g)^c(x^*, u^*,\alpha) \leq c( x_0, (x^*, u^*, \alpha))+ \varepsilon.
\end{equation}
Now,  denoting $h=g+c(\cdot, (x^*, u^*, \alpha))$, which is e-convex\footnote{ In \cite{RVP2011} is established on functions defined on $\R^n$, and it can be easily extended to the framework of locally convex spaces.}, we have
\begin{align*}
(f-g)^c (x^*, u^*, \alpha)&= \sup_{x \in X} \left \{ c(x,(x^*, u^*, \alpha))-f(x)+g(x) \right \} \\
&=\sup_{x \in X}  \left \{ h(x)-f(x) \right \}.
\end{align*}
Rewriting (\ref{eq7}),  $(x^*, u^*, \alpha) \in \partial_{c, \varepsilon} (f-g) (x_0)$ if and only if $ \langle   x_0,  u^* \rangle <  \alpha$ and
\begin{equation*}
\sup_{x \in X}  \left \{ h( x)-f( x) \right \} \leq h(x_0) -f(x_0)  + \varepsilon,
\end{equation*}
or, equivalently, according to Lemma \ref{lemma4},  $\langle   x_0,  u^* \rangle <  \alpha$ and
\begin{equation*}
f^c (y^*, v^*, \beta)+f( x_0) \leq h^c(y^*, v^*, \beta)  +h( x_0) + \varepsilon,
\end{equation*}
for all $(y^*, v^*, \beta) \in W$.
In virtue of Theorem \ref{theorem3} $ii)$, we conclude that, if $(x^*, u^*, \alpha) \in \partial_{c, \varepsilon} (f-g) (x_0)$, then 
\begin{equation} \label{eq8}
\partial_{c, \lambda}h (x_0) \subseteq \partial_{c,\lambda+ \varepsilon} f (x_0).
\end{equation}
To prove (\ref{eq6}), take $(y^*, v^*, \beta) \in \partial_{c, \lambda}g (x_0) $. We will see that $(x^*+y^*, u^*+v^*, \alpha+\beta ) \in \partial_{c, \lambda}h (x_0)$, and applying (\ref{eq8}), we will obtain (\ref{eq6}).
We have 
\begin{eqnarray} \label{eq9}
h^c(x^*+y^*, u^*+v^*, \alpha+\beta ) &=&\sup_{x \in X} \left \{ c(x,(x^*+y^*, u^*+v^*, \alpha+\beta ))-h(x) \right \} \nonumber \\ [0.2cm]
&\leq& \sup_{x \in X} \left \{c(x, (y^*, v^*, \beta))-g(x) \right \}=g^c(y^*, v^*, \beta).
\end{eqnarray}
Taking into account that $\left \langle x_0, u^* \right \rangle <\alpha$ and  $\left \langle x_0, v^* \right \rangle <\beta$, which implies that
\begin{equation*}
c( x_0,(x^*+y^*, u^*+v^*, \alpha+\beta ))=c(x_0,(x^*, u^*, \alpha))+c(x_0,(y^*, v^*, \beta)),
\end{equation*}
we obtain, by (\ref {eq9}) and moreover by  Theorem \ref{theorem3} $ii)$ applied to the fact that $(y^*, v^*, \beta) \in \partial_{c, \lambda}g ( x_0) $,
\begin{eqnarray*}
h( x_0)+h^c(x^*+y^*, u^*+v^*, \alpha+\beta ) &\leq & g(x_0) +c( x_0, (x^*, u^*, \alpha))+ g^c(y^*, v^*, \beta)\\ [0.2cm]
&\leq& c(x_0,(x^*, u^*, \alpha))+c(x_0,(y^*, v^*, \beta))+ \lambda \\[0.2cm]
&=& c(x_0,(x^*+y^*, u^*+v^*, \alpha+\beta ))+\lambda,
\end{eqnarray*}
and, again by Theorem \ref{theorem3} $ ii)$,  $(x^*+y^*, u^*+v^*, \alpha+\beta ) \in \partial_{c, \lambda}h (x_0)$.
\end{proof}

We present now a necessary condition for global optimality.
\begin{corollary}\label{cor1}
Let $f,g \colon X \to \Ramp$ be proper functions with $g$ e-convex. If $a \in X$ is a global minimizer of $f-g$, then 
$$\partial_{c,\varepsilon} g(a) \subseteq \partial_{c, \varepsilon} f(a),$$
for all $\varepsilon \geq 0$. 
\end{corollary}
\begin{proof}
 The inclusion is evident in the case $a \in \dom f$ but $a \notin \dom g$, where  $\inf_X \{ f(x)-g(x) \}=-\infty$.
Now, let us assume the alternative possibility, $\dom f \subseteq \dom g$.\\
It is straightforward that for a proper function $h \colon X \to \Ramp$, a point $a \in \dom h$ is a global minimizer of $h$ if and only if $(0_{X^*},0_{X^*}, \beta) \in \partial_c h(a)$, for all $\beta >0$. Hence, $a \in \dom (f-g)=\dom f$ is a global minimizer of the problem $\inf_X \{ f(x)-g(x) \}$  if and only if $(0_{X^*},0_{X^*}, \beta) \in \partial_c (f-g)(a)$, for all $\beta >0$. According to Theorem \ref{lemma5}, it implies that 
$$(0_{X^*},0_{X^*}, \beta) \in \partial_{c,\varepsilon} f(a)\overset{*}{-}\partial_{c, \varepsilon} g(a),$$
for all $\varepsilon \geq 0$, hence $(0_{X^*},0_{X^*}, \beta) +\partial_{c,\varepsilon} g(a) \subseteq \partial_{c, \varepsilon} f(a)$. Due to identity (\ref{relationepsilon}), this inclusion implies that $\partial_{\varepsilon} g(a) \subseteq \partial_{\varepsilon} f(a)$. Moreover, since $\dom f \subseteq \dom g$, we will have $\partial_{c,\varepsilon} g(a) \subseteq \partial_{c, \varepsilon} f(a).$
\end{proof}

%
Next example shows that the necessary condition in Corollary \ref{cor1} for the existence of a global minimizer is not sufficient.
\begin{example}\label{example1}
Let $f,g \colon \R^2 \to \Ramp$,
\begin{equation*} 
f(x,y) =\left\{
\begin{array}{ll}
x\ln{x \over y} , & \text{ if } (x,y) \in E,\\ [0.2cm]
0, & \text{ if } (x,y)=(0,0),\\ [0.2cm]
+\infty, & \text{ otherwise,}%
\end{array}%
\right.
\end{equation*}%
where $E=\{ (x,y) \in \R^2: 0< x \leq 1, x\geq y, y >0\}$, and 
\begin{equation*}
g(x,y) =\left\{
\begin{array}{ll}
0 , & \text{ if } x+y<2,\\ [0.2cm]
+\infty, & \text{ otherwise,}%
\end{array}%
\right.
\end{equation*}
The proof of the e-convexity of the function $f$ can be found in \cite[Ex.4.3]{FGRVP2020}.
The e-convexity of $g$ comes directly from the e-convexity of its effective domain.\\
We have that $$\inf_{(x,y) \in \R^2} \{ f(x,y)-g(x,y)\}=f(1,1)-g(1,1)=-\infty$$ and $(1,1)$ is a global minimizer of the DC problem but $(0,0)$ is not. We apply identity (\ref{relationepsilon}) to compute $\partial_{c,\varepsilon} g(0,0)$ and  $\partial_{c, \varepsilon} f(0,0)$, and
$$\partial_{c,\varepsilon} g(0,0)=\partial_{\varepsilon} g(0,0) \times \{(v^*,\alpha) \in \R^3: \dom g \subseteq H^<_{v^*, \alpha}\},$$
$$\partial_{c,\varepsilon} f(0,0)=\partial_{\varepsilon} f(0,0) \times \{(v^*,\alpha) \in \R^3: \dom f \subseteq H^<_{v^*, \alpha}\}.$$
In this case, $ \dom f \subseteq \dom g$, and  $\dom f \subseteq H^<_{v^*, \alpha}$ whenever  $\dom g \subseteq H^<_{v^*, \alpha }$, which means that the only thing to prove for having the inclusion 

$$\partial_{c,\varepsilon} g(0,0) \subseteq \partial_{c, \varepsilon} f(0,0)$$
 is $\partial_{\varepsilon} g(0,0) \subseteq \partial_{\varepsilon} f(0,0)$. \\
 Then, let us take $(u,v) \in \partial_{\varepsilon} g(0,0)$. It means that $ux+vy \leq \varepsilon$, for all $(x,y) \in \dom g$. We will have, for all $(x,y) \in E$,
 $$ x\ln {x \over y} \geq 0 \geq ux+vy-\varepsilon,$$
and $0 \geq u0+v0-\varepsilon$, then
$$f(x,y) \geq f(0,0)+ux+vy-\varepsilon, \text{ for all } (x,y) \in \R^2$$
and $(u,v) \in \partial_{\varepsilon} f(0,0)$.
\end{example}

Finally, we close this section presenting a necessary condition for  $\varepsilon$-minimizers.
\begin{corollary}\label{thm:Local_minima}
Let $f,g \colon X \to \Ramp$ be proper functions with $g$ e-convex. If $a \in X$ is an $\varepsilon$-minimizer of $f-g$ then, for all $\lambda \geq 0$,
$$\partial_{c,\lambda} g(a) \subseteq \partial_{c,\varepsilon+\lambda} f(a).$$
\end{corollary}
\begin{proof}

 The inclusion is clear in the case $\dom f \not \subseteq \dom g$, since the set of $\varepsilon$-minimizers is $\{ a\in X:a \in \dom f \setminus \dom g\}$, (in fact there is no difference between $\varepsilon$-minimizers and global minimizers), and for these points $\partial_{c,\lambda} g(a)=\emptyset$.
Now, let us assume $\dom f \subseteq \dom g$.\\
For a proper function $h \colon X \to \Ramp$, a point $a \in \dom h$ is an $\varepsilon$-minimizer of $h$ if and only if $(0_{X^*},0_{X^*}, \beta) \in \partial_{c,\varepsilon} h(a)$, for all $\beta >0$.  Hence, $a \in \dom (f-g)=\dom f$ is an $\varepsilon$-minimizer of the problem $\inf_X \{ f(x)-g(x) \}$  if and only if $(0_{X^*},0_{X^*}, \beta) \in \partial_{c,\varepsilon} (f-g)(a)$, for all $\beta >0$. According to Theorem \ref{lemma5}, it implies that 
$$(0_{X^*},0_{X^*}, \beta) \in \partial_{c,\varepsilon+\lambda} f(a)\overset{*}{-}\partial_{c, \varepsilon} g(a),$$
for all $\lambda \geq 0$, and the proof continues along the lines of the proof of Corollary \ref{cor1}.
\end{proof}

\begin{remark}
Example \ref{example1} shows that this necessary condition is not sufficient: for all $\varepsilon \geq 0$, $(0,0)$ in not an $\varepsilon$-minimizer but, for all $\lambda \geq 0$,
$$  \partial_{c,\lambda} g(0,0) \subseteq \partial_{c, \lambda} f(0,0) \subseteq \partial_{c, \varepsilon + \lambda} f(0,0).$$
\end{remark}
\begin{remark}
Local optimality necessary or sufficient conditions for DC problems where both functions are proper and convex (although convexity for $f$ is not essential) can be found in \cite{HU1988} and \cite {D2003}. Again they are expressed in terms of $\varepsilon$-subdifferential sets. 

A characterization of local optimality in the finite dimensional context  given in \cite[Th.~4.3]{BL2010}, assumes that the functions $f$ and $g$ are convex and lsc, and it is
 $$ \partial_{\varepsilon} g(a) \subseteq \bigcup_{\sigma \in [0, \varepsilon]} \left \lbrace{\partial_{\varepsilon} f(a)+B \Biggl ( 0_n, {{\varepsilon-\sigma} \over \eta}} \Biggr ) \right \rbrace, \text{   for all } \varepsilon >0,$$
where $B(x,\nu)$ stands for the ball of radius $\nu>0$ centered at $x \in \R^n$, and $\eta >0$ is any scalar for which $a$ is an optimum in the ball $B(a, \eta)$. It could also have its counterpart for e-convex functions and $\varepsilon$-$c$-subdifferential sets. Nevertheless, as it can be observed in \cite[Th.~4.3 Proof]{BL2010}), some further constraint qualifications are needed to split the $\varepsilon$-$c$-subdifferential of the sum of two e-convex functions. For more information on this, we encourage the reader to check \cite[Th.~11]{FVR2012}.
\end{remark}
%

\section{Conclusions and future research}
\label{sec:6}

Throughout this manuscript we have exploited the main properties that a subdifferential defined via a generalized conjugation scheme satisfies. With the purpose of generalizing some results from \cite{Z2002}, we have investigated the role of the $\varepsilon$-directional derivative and we have taken an insight on how the support function of the $c$-subdifferential may be derived.

As a theoretical application of the $c$-subdifferential, we have focused on the development of global optimality and $\varepsilon$-optimality conditions for DC problems. For problems whose objective function reads as the difference of two e-convex functions, which can be denoted by eDC problems, we have adapted well-known results from J.B. Hiriart-Hurruty in \cite{HU1988}, which turns out to give necessary but not sufficient conditions via $c$-subdifferentials.

Throughout the manuscript we have pointed out some open issues that, from our perspective, go beyond the scope of the paper and deserve to be studied thoroughly as a future research. We conclude the paper mentioning the application of the $c$-subdifferential in the study of Toland-Singer duality. This type of duality has become quite popular in the community of DC programming when the involved functions are proper convex and lsc, so we expect the $c$-subdifferential to lead the duality theory of eDC problems. 

\section*{Disclosure statement}
The authors declare that they have no conflict of interest.

\section*{Funding}
Research partially supported by MICIIN of Spain and ERDF of EU, Grant PGC2018 097960-B-C22.

\bibliographystyle{spmpsci}
\bibliography{biblio}

\end{document}